\documentclass[12pt]{article}

\usepackage{graphicx} % Add graphics capabilities
\usepackage{amsmath} % Better maths support & more symbols
\usepackage{amssymb,amsthm}
\usepackage{color}
\usepackage{tikz}
\usepackage{pgfplots}

\newcommand{\R}{{\mathbb R}}\newcommand{\N}{{\mathbb N}}
\newcommand{\Res}{\mathrm{Res}}
\newcommand{\Z}{{\mathbb Z}}\newcommand{\C}{{\mathbb C}}

\newcommand{\jap}[1]{\left\langle#1\right\rangle}

\let\epsilon\varepsilon
\let\hat\widehat

\newtheorem{theorem}{Theorem}[section]\newtheorem{lemma}[theorem]{Lemma}

\newtheorem{remark}[theorem]{Remark}

\title{Approximate Peregrine solitons in dispersive nonlinear wave equations}
\author{Guido Schneider, Nils Thorin \\
{\small
Institut f\"ur Analysis, Dynamik und Modellierung,} \\ {\small Universit\"at Stuttgart, Pfaffenwaldring 57, } \\ {\small 70569 Stuttgart, Germany}}

\begin{document}

\maketitle

\begin{abstract}
The purpose of this short note is to explain how the existing results on the validity of the NLS approximation can be extended from Sobolev spaces $ H^s(\R) $ to the 
spaces of functions $ u = v + w $ where $ v \in H_{{per}}^s $ and $ w \in H^s(\R) $.
This allows us to use the Peregrine solution of the NLS equation to find freak or rogue wave dynamics in more complicated systems.
\end{abstract}

\section{Introduction}

The NLS equation 
$$ 
i \partial_T A = \nu_1 \partial_X^2 A + \nu_2  A |A|^2, \qquad (\nu_1,\nu_2 \in \R)
$$
can be derived by multiple scaling perturbation analysis for the description of  slow modulations in time and space of the envelope of spatially and temporarily oscillating wave packets, as they appear in nonlinear optics, water wave theory, plasma
physics, waves in DNA, Bose-Einstein condensates, etc.

As an example we consider  the cubic Klein-Gordon equation
\begin{equation} \label{eq:3a1}
\partial_t^2 u = \partial_x^2 u - u + u^3,
\end{equation}
for $ x,t,u(x,t) \in \R $. We make the ansatz 
\begin{equation} \label{ansatz1}
	u(x,t) = \varepsilon \Psi_{NLS}(x,t) = \varepsilon A(\varepsilon(x-c_g t),\varepsilon^2 t) e^{i(k_0 x-\omega_0 t)} + c.c.
\end{equation}
with spatial and temporal wave number $ k_0 > 0 $ and $ \omega_0 \in \R $, 
with linear group velocity $ c_g \in \R $,
with envelope function 
$ A = A(X,T) \in \C $, scaled variables $X=\varepsilon(x-c_g t)$, $T=\varepsilon^2 t$,
and small perturbation parameter $ 0 < \varepsilon \ll 1 $.
By inserting the ansatz in \eqref{eq:3a1}
and by equating the coefficients in front of $  \varepsilon^j e^{i(k_0 x-\omega_0 t)} $ 
to zero for $ j = 1,2,3 $, we find the linear dispersion relation $ \omega_0^2 = k_0^2 + 1 $, the 
linear group velocity, here $ c_g = k_0/\omega_0  $, and finally 
the NLS equation
\begin{equation} \label{nls5}
2 i \omega_0 \partial_{T} A = (c_g^2-1)\partial_{X}^2 A - 3 A |A|^2.
\end{equation}

\begin{figure}[htbp] \centering

  \includegraphics[width=14cm]{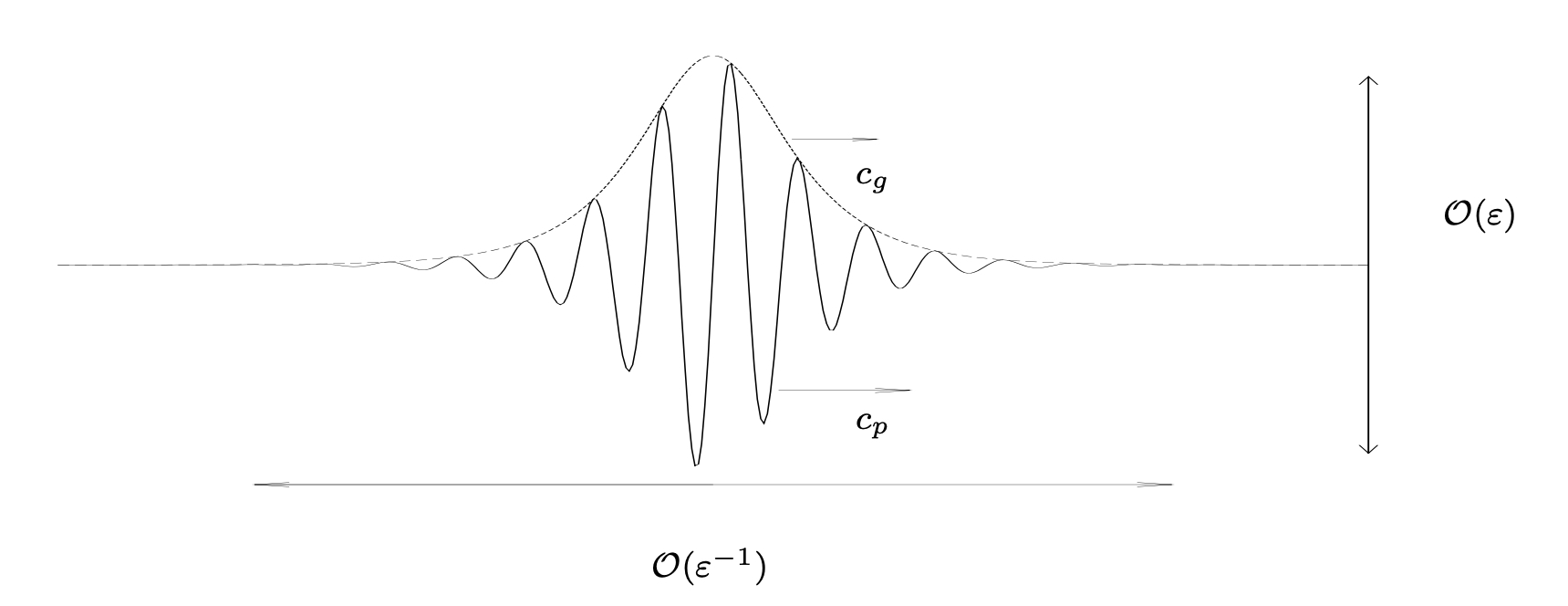} 
\medskip

  \caption{{\small  A modulating 
pulse described by the NLS equation. The envelope, advancing with 
group velocity $ c_g $ in the laboratory frame,  modulates the underlying 
carrier wave  $ e^{i (k_0 x - \omega_0 t)} $, advancing with 
phase velocity $ c_p $. The envelope evolves approximately as a solution of 
the NLS equation.}}
  \label{nn1}
\end{figure}

Several NLS approximation results have been established in the last decades.
For the above cubic Klein-Gordon equation we have for instance \cite{KSM92}: 
\medskip

{\bf Theorem.}
 Let $A \in C ([0,T_0], H^{4})$ 
  be a
  solution of the NLS equation \eqref{nls5}.  Then there exists an
  $\varepsilon_0 > 0$ and a $C>0$ such that for all $\varepsilon \in (0,\varepsilon_0)$
  there are solutions $u$ of the original system \eqref{eq:3a1} which
  can be approximated by $\varepsilon \Psi_{NLS} $ 
 with
\[
\sup_{t\in [0,T_{0}/\varepsilon^2]} \|u(t)-\varepsilon \Psi_{NLS}(t)\|_{H^{1}} < C\varepsilon^{3/2}.
\]
The NLS equation was first derived by Zakharov in 1968 for the water wave problem \cite{Za68}.
In general,
error estimates are non-trivial for quadratic nonlinearities since solutions of order $ \mathcal{O}(\varepsilon) $ 
have to be controlled on a time scale of order $ \mathcal{O}(1/\varepsilon^2) $.
The results are based on near identity change of variables using the oscillatory character 
of the  $ \mathcal{O}(\varepsilon) $-terms in the error equations, cf. \cite{Ka87}.
Resonant terms can however lead to non-approximation results. There are rigorous counter examples 
that the NLS approximation fails to make correct predictions, see \cite{SSZ15}.
An  NLS approximation result for the water wave problem  can be found for instance in \cite{Du21}.

It is the purpose of this short note to explain how the existing approximation results can be extended from Sobolev spaces $ H^s(\R) $ to the 
spaces of functions $ u = v + w $ where $ v \in H_{{per}}^s $ and $ w \in H^s(\R) $.
This allows us to use the Peregrine solution of the NLS equation to find rogue wave dynamics in more complicated systems, too.
Rogue waves, also known as freak waves or killer waves, are large and unpredictable surface waves that can be extremely dangerous to ships. 
Their height is significantly greater than that of average waves, and they seem to appear out of nowhere. These phenomena are not confined to water waves; they also occur in liquid helium, non-linear optics, and microwave cavities, see \cite{roguebook}.

The NLS equation in its normalized form 
$$
i \partial_\tau \psi+ \frac{1}{2}  \partial _\xi^2\psi+ |\psi|^2 \psi = 0
$$
possesses so called 
Peregrine solutions
\begin{equation}
A(\xi, \tau) =  \left[ 1-\frac{4 (1 + 2 i \tau)}{1+4 \xi^2 + 4 \tau^2}  \right] e^{i \tau}
\end{equation}  
which were discovered by Peregrine \cite{Pe83} and are prototype of a rogue wave. See 
Figure \ref{perefig1}.

\begin{figure}
  \centering
  \begin{minipage}[b]{0.45\textwidth}\label{peremod}
    \includegraphics[width=\textwidth]{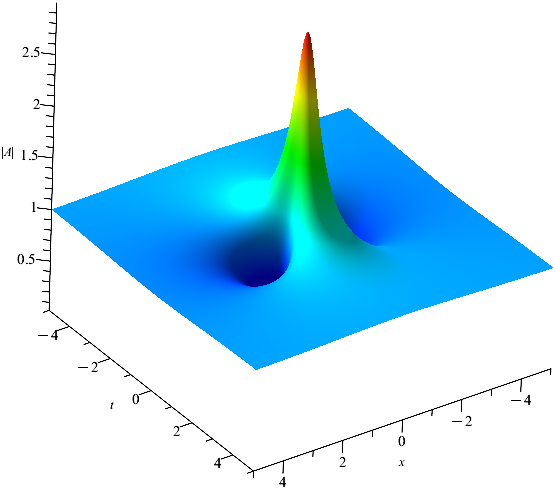}
  \end{minipage}
  \hfill
  \begin{minipage}[b]{0.45\textwidth}\label{perereal}
    \includegraphics[width=\textwidth]{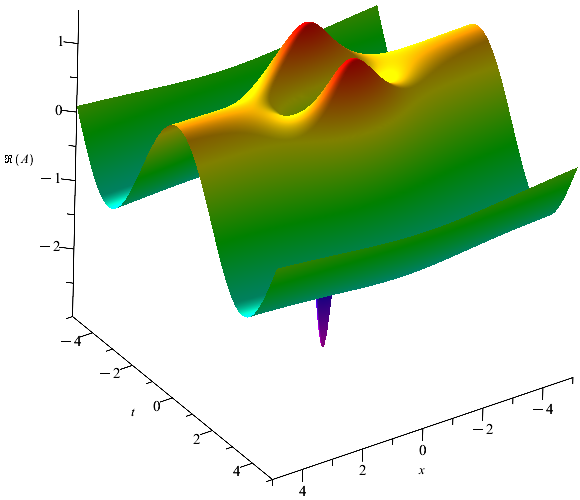}
  \end{minipage}
  \caption{Absolute value (left) and real part (right) of the Peregrine soliton plotted as a function of space and time.}
  \label{perefig1}
\end{figure}

The existing NLS approximation results do not apply since this solution does not decay to zero
for $ |\xi| \to \infty$. However, it can be written as
$$
A(\xi, \tau) = e^{i \tau} + B(\xi,\tau),
$$
with $ B(\cdot,\tau) \in W^{s,p}(\R) $
for all $ p \in [1,\infty]$ and $ s \in \N_0$.
In a similar manner we split the ansatz, defined in \eqref{ansatz1}, into
$$
\varepsilon \Psi_{NLS}(x,t) = \varepsilon \Psi_v(x,t) + \varepsilon \Psi_w(x,t),
$$
with 
$$ \varepsilon \Psi_v(x,t) = \varepsilon \Psi_v(x+ 2 \pi/k_0,t) , \quad
 \varepsilon \Psi_v(\cdot,t) \in W^{s,p}_{per},  $$ 
 and
 $$
 \varepsilon \Psi_w(\cdot,t) \in W^{s,p}(\R) .$$
In a similar manner we split the solution $ u $ of the cubic Klein-Gordon equation into 
$$ 
u = v + w ,
$$ 
with 
$$ 
v(x,t) = v(x+ 2 \pi/k_0,t), \quad v(\cdot,t) \in W^{s,p}_{per} ,
\quad \textrm{and} \quad
w(\cdot,t) \in W^{s,p}(\R). 
$$
The new variables $ v $ and $ w $ satisfy 
\begin{eqnarray} \label{veq}
\partial_t^2 v & = & \partial_x^2 v - v + v^3 ,
\\
\partial_t^2 w & = & \partial_x^2 w - w + 3 v^2 w + 3 v w^2 + w^3 .\label{weq}
\end{eqnarray}
Then we have the following approximation result 
\begin{theorem} \label{thmain}
Fix $s > 1/2$ and
let $ A = A_v + A_w$, with 
$$(A_v, A_w) \in C([0,T_0],\C\times H^{s+4}(\R)) ,$$ 
be a solution of the NLS equation \eqref{nls5}.
Then there exist $ \varepsilon_0 > 0 $ and $ C > 0 $ such that for all
$ \varepsilon \in (0, \varepsilon_0) $ there are solutions 
$$(v,w) \in C([0,T],H^s_{per}\times H^s(\R))$$ 
of the KG equation \eqref{eq:3a1} with
$$
\sup_{t \in [0,T_0/ \varepsilon^2]} \| (v,w)(\cdot,t) - (\varepsilon\Psi_v,\varepsilon\Psi_w)(\cdot,t) \|_{H^s_{per} \times H^s} \leq C \varepsilon^{3/2}.
$$
\end{theorem}
Since we have a cubic nonlinearity the application of Gronwall's inequality 
to the variation of constant formula is sufficient for getting an $ \mathcal{O}(1) $-bound 
for $ R_v $ and $ R_w $ on the long $ \mathcal{O}(1/\varepsilon^2) $-time scale.
We expect that most of the existing NLS approximation theory (the handling of quadratic terms by normal form 
transformations, cf. \cite{Ka87}, the handling of resonances, cf. \cite{Schn05}, and quasilinear systems, cf. \cite{Du17,WC17,DH18}) can be transferred to these 
function spaces.  
There are other interesting solutions of the NLS equation having the same 
analytic properties as the Peregrine solutions. Hence the above
 approximation result also applies to so called 
 Kuznetsov-Ma solitons and higher order Peregrine solitons. See Section \ref{secdisc}.
 \medskip 
 
 \noindent
{\bf Acknowledgement.}
The paper 
is partially supported by the Deutsche Forschungsgemeinschaft DFG through the SFB 1173 ``Wave phenomena" Project-ID 258734477.

\section{The functional analytic set-up}

In this section we introduce the functional analytic set-up which we need to prove 
Theorem \ref{thmain}.

{\bf A)} In order to solve \eqref{veq} we use the space 
$$ 
H^s_{{per,k_0}} = \{ v : \R \to \R : v(x) = v(x+ 2 \pi/k_0), \| v \|_{H^s_{{per}}} < \infty\}
$$ 
for a $ k_0 > 0 $, with 
$$ 
\|{v}\|_{H^s_{{per}}} = 
\| \hat{v} \|_{\ell^2_s} = 
\|({\left\langle k \right\rangle}^{s/2}\hat{v}_k)_{k\in\Z}\|_{\ell^2} = \left(\sum_{k\in\Z} |\widehat{v}_k|^2 (1+k^2)\right)^{1/2},
$$
where $ v(x) = \frac{k_0}{2 \pi}\sum_{k \in \Z} \widehat{v}_k e^{ik k_0 x} $,
i.e.,  $(\hat{v}_k)_{k\in\Z}$ is the  discrete Fourier transform of $v$. 
In the following we write $ 
H^s_{{per}} $ instead of $ 
H^s_{{per,k_0}} $ if no confusion is possible. 
The space $ H^s_{{per}} $ is closed under pointwise multiplication
if $ s > 1/2 $ and can be embedded in $ C^0_{b,unif} $ if $ s > 1/2 $.
In detail,
for each $s>1/2$ there exists a $C>0$ such that for all $v_1, v_2 \in H^s_{{per}}$ we have
	\begin{equation*}
		\|v_1 v_2 \|_{H^s_{{per}}} \leq C\|v_1\|_{H^s_{{per}}}\|v_2\|_{H^s_{{per}}}.
	\end{equation*}

{\bf B)} In order to solve \eqref{weq} we use the space 
$$ 
H^s = H^s(\R) = \{ w : \R \to \R :  \| w \|_{H^s} < \infty\},
$$
with 
$$ 
\|w\|_{H^s} = \| \widehat{w} \|_{L^2_s} = \|\left\langle \cdot \right\rangle^{s/2} \widehat{w}(\cdot) \|_{L^2} = \left(\int (1+ k^2)^s |\widehat{w}(k)|^2 dk\right)^{1/2},
$$
i.e., $\hat{w}$ is the continuous  Fourier transform of $w$.
The space $ H^s(\R) $ is closed under pointwise multiplication
if $ s > 1/2 $ and can be embedded in $ C^0_{b,unif} $ if $ s > 1/2 $.
In detail,
for each $s>1/2$ there exists a $C>0$ such that for all $w_1, w_2 \in H^s(\R) $ we have
	\begin{equation*}
		\|w_1 w_2 \|_{H^s} \leq C\|w_1\|_{H^s}\|w_2\|_{H^s}.
	\end{equation*}

{\bf C)} In \eqref{weq} there are multiplications between $ v \in H^s_{{per}} $
and $ w \in H^s $, too.
For $ s > 1/2 $ these can be estimated with 
	\begin{eqnarray*}
		\|{v w}\|_{H^s} &= & \|{{\left\langle k \right\rangle}^{s/2} \hat{v w}(k)}\|_{L^2(dk)} \leq \left\|{\int {\left\langle k \right\rangle}^{s/2} |\hat{v}(k-\ell)\hat{w}(\ell)| d\ell}\right\|_{L^2(dk)}\\
		&\leq& C \left\|{\int {\left\langle k-\ell \right\rangle}^{s/2}|\hat{v}(k-\ell)\hat{w}(\ell)| d\ell +  \int {\left\langle \ell \right\rangle}^{s/2}|\hat{v}(k-\ell)\hat{w}(\ell)| d\ell}\right\|_{L^2}\\
		&\leq & C \|{|{\left\langle \cdot \right\rangle}^{s/2}\hat{v}| * |\hat{w}| + |\hat{v}| * |{\left\langle \cdot \right\rangle}^{s/2}\hat{w}|}\|_{L^2}\\
		&\leq &C\left(\|{{\left\langle \cdot \right\rangle}^{s/2} \hat{v}}\|_{\ell^2} \|{\hat{w}}\|_{L^1} + \|{\hat{v}}\|_{\ell^1}\|{\left\langle \cdot \right\rangle}^{s/2}\hat{w}\|_{L^2}\right)\\
		& = & C(\|{\hat{v}}\|_{\ell^2_{s}}\|{\hat{w}}\|_{L^1} + \|{\hat{v}_i}\|_{\ell^1_{0}}\|{\hat{w}}\|_{L^2_s})\\
		& \leq & 2C\|{v}\|_{H^s_{{per}}}\|{w}\|_{H^s},
	\end{eqnarray*}
	where we  used the identity $${\left\langle k \right\rangle}^s\leq C(\jap{k-\ell}^s + \jap{\ell}^s),$$  in the second inequality, Young's inequality in the fourth inequality, and Sobolev's embedding theorem in the last inequality.

\begin{remark}{\rm 
The goal of this paper can also be reformulated as follows.
It is the purpose of this short note is to explain how the existing results 
about the validity of the NLS approximation
can be extended from Sobolev spaces to spaces formally defined by
\begin{equation*} 
    M^s = \{u = v+w :  v\in {H^s_{{per}}}, w\in H^s\},
\end{equation*}
with the corresponding norm
\begin{equation*}
	\|u\|_{M^s} =  \| v \|_{H^s_{{per}}} + \|w\|_{H^s}.
\end{equation*}
As a direct consequence of the above estimates 
the spaces $ M^s $ are algebras for $s>1/2$, i.e.,
for each $s>1/2$ there exists a $C>0$ such that for all $u_1, u_2 \in {M^s}$ we have
	\begin{equation*}
		\|u_1 u_2 \|_{M^s} \leq C\|u_1\|_{M^s}\|u_2\|_{M^s}.
	\end{equation*}}
\end{remark}

\section{Construction of an improved approximation}

We follow the lines of \cite{KSM92} and consider the improved approximation
\begin{eqnarray*}
u(x,t) = \varepsilon \Psi(x,t) & = & \varepsilon A(\varepsilon(x-c_g t),\varepsilon^2 t) e^{i(k_0 x-\omega_0 t)}  \\ && + 
 \varepsilon^3 A_3(\varepsilon(x-c_g t),\varepsilon^2 t) e^{3 i(k_0 x-\omega_0 t)}
+ c.c..
\end{eqnarray*}
By this choice we have that the so called residual
$$ 
\textrm{Res}  =  - \partial_t^2 u + \partial_x^2 u - u + u^3,
$$ 
i.e., the terms which do not cancel after inserting the approximation into the equation, is formally of order $ \mathcal{O}(\varepsilon^4) $.
For estimating the approximation and the residual we split $ A = A_v + A_w $
into a spatially constant part $ A_v $ and into a spatially decaying part $ A_w $.
Similarly, we split $ A_3 = A_{3,v} + A_{3,w} $. Thus the improved approximation is then of the 
form 
$$ 
u(x,t) = \varepsilon \Psi(x,t) = v(x,t)+w(x,t) = \varepsilon \Psi_v(x,t) + \varepsilon \Psi_w(x,t),
$$ 
where 
\begin{eqnarray*}
 \varepsilon \Psi_v(x,t) & = & \varepsilon A_v(\varepsilon(x-c_g t),\varepsilon^2 t) e^{i(k_0 x-\omega_0 t)}  \\ && + 
 \varepsilon^3 A_{3,v}(\varepsilon(x-c_g t),\varepsilon^2 t) e^{3 i(k_0 x-\omega_0 t)}
+ c.c. ,\\
 \varepsilon \Psi_w(x,t) & = & \varepsilon A_w(\varepsilon(x-c_g t),\varepsilon^2 t) e^{i(k_0 x-\omega_0 t)}  \\ && + 
 \varepsilon^3 A_{3,w}(\varepsilon(x-c_g t),\varepsilon^2 t) e^{3 i(k_0 x-\omega_0 t)}
+ c.c. .
\end{eqnarray*}
By inserting this ansatz into \eqref{eq:3a1}
and by equating the coefficients in front of $  \varepsilon^j e^{i(k_0 x-\omega_0 t)} $ to zero, for $ j = 1,2,3 $, we find, as in the introduction, the linear dispersion relation $ \omega_0^2 = k_0^2 + 1 $, the 
linear group velocity $ c_g = k_0/\omega_0  $, and 
the NLS equation
\begin{eqnarray} \label{Aveq}
2 i \omega_0 \partial_{T} A_v & = & - 3 A_v |A_v|^2, \quad\\
2 i \omega_0 \partial_{T} A_w & = & (c_g^2-1)\partial_{X}^2 A_w - 3 A_v^2 \overline{A}_w -  6 A_v A_w \overline{A}_v 
\label{Aweq}
 \\ && - 3 A_w^2 \overline{A}_v -  6 A_w A_v \overline{A}_w - 3 A_w^2 \overline{A}_w .  \nonumber
\end{eqnarray}

By equating the coefficient in front of $  \varepsilon^3 e^{3 i(k_0 x-\omega_0 t)} $ to zero we finally find 
\begin{eqnarray} \label{A3veq}
0 & = &  (9 k_2 - 9 \omega_0^2 + 1) A_{3,v} - A_v^3 , \\ 
0 & = &  (9 k_2 - 9 \omega_0^2 + 1) A_{3,w} - 3 A_v^2 A_w - 3 A_v A_w^2 - A_w^3.
\label{A3weq}
\end{eqnarray} 
Similarly to the splitting  of $ u = v + w $ 
we split the residual $ \textrm{Res} = \textrm{Res}_v + \textrm{Res}_w $, into
\begin{eqnarray} \label{vres}
 \textrm{Res}_v(v) & = & - \partial_t^2 v + \partial_x^2 v - v + v^3 ,
\\
 \textrm{Res}_w(v,w) & = & - \partial_t^2 w + \partial_x^2 w - w + 3 v^2 w + 3 v w^2 + w^3 .\label{wres}
\end{eqnarray}

\section{Estimates for the residual}

We have 
\begin{lemma}
Assume $ s_A -3 \geq s \geq 1 $.
Let $ (A_{v},A_{w}) \in C([0,T_0],\C \times H^{s_A}) $ be a solution of the NLS equation 
\eqref{Aveq}-\eqref{Aweq} and let $ (A_{3,v},A_{3,w}) $ be defined as solution of
\eqref{A3veq}-\eqref{A3weq}. Then there exist $ \varepsilon_0 > 0 $ and $ C > 0 $ such that 
for all $ \varepsilon \in (0,\varepsilon_0) $ we have 
$$ 
\sup_{t \in [0,T_0/\varepsilon^2]} \|\textrm{\rm Res}_v( \varepsilon \Psi_v) \|_{H^s_{per}}
\leq C \varepsilon^4 
$$ 
and 
$$ 
\sup_{t \in [0,T_0/\varepsilon^2]} \|\textrm{\rm Res}_w( \varepsilon \Psi_v, \varepsilon \Psi_w) \|_{H^s}
\leq C \varepsilon^{7/2} .
$$
\end{lemma}
\noindent
{\bf Proof.}
By construction we have that the remaining terms in the residuals 
are formally of  order $ \mathcal{O}(\varepsilon^4) $. The loss
of $ \varepsilon^{-1/2} $ for $ \textrm{Res}_w( \varepsilon \Psi_v, \varepsilon \Psi_w) $ comes from the scaling properties
of the $ L^2(\R) $-norm w.r.t. the scaling $ X = \varepsilon x $. 
We refrain from giving the detailed estimates which are completely straightforward, cf. \cite[Section 11.2]{SU17book}. 
\qed

\section{The error estimates}

The subsequent error estimates adapt the proof in \cite[Section 11.2]{SU17book}. We introduce error functions $\varepsilon^{\beta} R_v$ and $\varepsilon^{\beta} R_w$ with $\beta =  3/2$ by:
$$ 
v = \varepsilon \Psi_v + \varepsilon^{\beta} R_v , \qquad 
w = \varepsilon \Psi_w + \varepsilon^{\beta} R_w,
$$
where
$$ 
R_v(x,t) = R_v(x+ 2 \pi/k_0,t), \qquad R_w(\cdot,t) \in L^2(\R) .
$$
The error functions $ R _v $ and $ R_w $ satisfy
\begin{eqnarray} \label{huerrv}
\partial_t^2 R_v & = & \partial_x^2 R_v - R_v + \varepsilon^2f_v , \\ 
\partial_t^2 R_w & = & \partial_x^2 R_w - R_w + \varepsilon^2f_w, \label{huerrw}
\end{eqnarray}
where 
\begin{eqnarray*}
\varepsilon^2 f_v & = &  3 \varepsilon^2 \Psi_v^2 R_v + 3  \varepsilon^{1+\beta}
\Psi_v R_v^2 + \varepsilon^{2 \beta} R_v^3 + \varepsilon^{- \beta} \textrm{Res}_v(\varepsilon \Psi_v),\\
\varepsilon^2 f_w & = &
 3 \varepsilon^2 (\Psi_v+ \Psi_w)^2 R_w 
+ 3 \varepsilon^2  \Psi_w^2 R_v  + 6 \varepsilon^2 \Psi_v \Psi_w  R_v \\ 
&& + 3  \varepsilon^{1+\beta} (\Psi_v+ \Psi_w) R_w^2 
+ 3  \varepsilon^{1+\beta} \Psi_w R_v^2 + 6  \varepsilon^{1+\beta} \Psi_w R_v R_w
\\ && 
+ 3 \varepsilon^{2 \beta} R_v^2 R_w + 3 \varepsilon^{2 \beta} R_v R_w^2 + \varepsilon^{2 \beta} R_w^3 
+ \varepsilon^{- \beta} \textrm{Res}_w(\varepsilon \Psi_v,\varepsilon \Psi_w).
\end{eqnarray*}
The associated equations 
 in 
 Fourier space, namely
 \begin{eqnarray*}
\partial_t^2 \widehat{R}_q & = & -   \omega^2 \widehat{R}_q  +\varepsilon^2\widehat{f}_q ,
\end{eqnarray*}
with $ q = v,w $ and $ \omega(k) = \sqrt{k^2 + 1} $, are written as a first order system
 \begin{eqnarray*}
\partial_t \widehat R_{q,1} &= & i \omega \widehat R_{q,2} , \\
\partial_t \widehat R_{q,2} &= & i \omega \widehat R_{q,1} + \varepsilon^2  \frac{1}{i \omega}
\widehat f_q.
\end{eqnarray*}
This system is abbreviated  as
\[
\partial_t \widehat{\mathcal{R}}_q(k,t) = \widehat{\Lambda}(k) \widehat{\mathcal{R}}_q(k,t) + \varepsilon^2 \widehat{F}_q(k,t),
\]
with 
$$ 
\widehat{\Lambda}(k) = \left(\begin{array}{cc} 0 & i \omega(k) \\  i \omega(k) & 0 \end{array} \right), \qquad
 \widehat{F}_q(k,t)= \left(\begin{array}{c} 0  \\  \frac{1}{i \omega} \widehat{f}_q(k,t)\end{array} \right).
$$ 
We use the variation of constant formula
\[
\widehat{\mathcal{R}}_q (k,t) = e^{t \widehat{\Lambda}(k) } \widehat{\mathcal{R}}_q(k,0) + \varepsilon^2 \int^t_0 
e^{(t-\tau)\widehat{\Lambda}(k)
  } \widehat{F}_q(k,\tau) d \tau
\]
to estimate the solutions of this system.
We start with the estimate for the linear semigroup.
\begin{lemma}
The semigroup $(e^{t \widehat{\Lambda}(k) })_{t\geq 0}$ is uniformly bounded in every $L^2_{s}$ and every $ \ell^2_{s} $, 
i.e., for every $ s \geq 0 $
there exists a $ C > 0 $ such that 
$$\sup_{t\in \R} \|e^{t \widehat{\Lambda}(k) }\|_{L^2_{s} \to L^2_{s}}+ \sup_{t\in \R} \|e^{t \widehat{\Lambda}(k) }\|_{\ell^2_{s} \to  \ell^2_{s}} \leq C.$$
\end{lemma}
\noindent 
{\bf Proof.}
In the following $ X^{s} $ denotes $L^2_{s}$ or $ \ell^2_{s} $.
We have $\widehat{\Lambda}(k) = S \widehat{D}(k)S^{-1}$ and as a consequence
$
e^{t  \widehat{\Lambda}(k) } = S e^{t \widehat{D}(k) } S^{-1} $
where  
\[
S = \begin{pmatrix}1 & 1
    \\ 1 & -1 \end{pmatrix} \qquad {\rm  and} \qquad \widehat{D}(k) = \begin{pmatrix} i \omega(k) & 0 \\ 0 & -i \omega(k)\end{pmatrix}.
\]
The estimate follows from
\[
\|e^{t \widehat{\Lambda}(k)} \widehat{u}\|_{X^{s}} \leq \sup_{k\in \R} \|e^{t \widehat{\Lambda}(k)
  }\|_{\C^2 \to \C^2} \|\widehat{u}\|_{X^{s}},
\]
and
\begin{align*}
\sup_{k\in \R} \|e^{t \widehat{\Lambda}(k) }\|_{\C^2 \to \C^2}
&\leq \|S\|_{\C^2 \to \C^2} \sup_{k\in \R} \|e^{t \widehat{D}(k) }\|_{\C^2 \to \C^2}
\cdot \|S^{-1}\|_{\C^2 \to \C^2} \\
&\leq \|S\|_{\C^2 \to \C^2}  \|S^{-1}\|_{\C^2 \to \C^2} < \infty.
\end{align*}
\qed

For the nonlinear terms we have 
\begin{lemma} \label{sync236}
For every ${s}>  1/2 $ there is a $C > 0$ such that
for all $ \varepsilon \in (0,1] $ we have 
\begin{eqnarray*}
\|\widehat{F}_v\|_{\ell^2_{s}} & \leq & C \left( \|\widehat{\mathcal{R}}_v\|_{\ell^2_{s}} + \varepsilon^{\beta - 1}
  \|\widehat{\mathcal{R}}_v\|^2_{\ell^2_{s}} + \varepsilon^{2\beta - 2} \|\widehat{\mathcal{R}}_v|^3_{\ell^2_{s}} + 1\right), \\ 
  \|\widehat{F}_w\|_{L^2_{s}} & \leq & 
   C \left( \|\widehat{\mathcal{R}}_v\|_{\ell^2_{s}} + \|\widehat{\mathcal{R}}_w\|_{L^2_{s}} + \varepsilon^{\beta - 1}
(\|\widehat{\mathcal{R}}_v\|_{\ell^2_{s}} + \|\widehat{\mathcal{R}}_w\|_{L^2_{s}})^2\right.\\&&\left.+\varepsilon^{2\beta - 2} (\|\widehat{\mathcal{R}}_v\|_{\ell^2_{s}} + \|\widehat{\mathcal{R}}_w\|_{L^2_{s}})^3 + 1\right).
\end{eqnarray*}
\end{lemma}

\begin{proof}
Let $q\in\{v,w\}$, and $X^s_q, Y^s_q$ be $\ell^2_s,\ell^1_s$ for $q=v$ and $L^2_s, L^1_s$ for $q=w$.
Then the estimates follow from
\begin{align*}& 
\Big\|\frac{1}{\omega} \widehat u\Big\|_{X^s_q} 
\leq \|\widehat u\|_{X^s_q}, \quad
\|\varepsilon^{-\beta} \widehat{\Res_q}\|_{X^s_q} 
\leq C,
\end{align*}
as well as from
\begin{align*}
\left\|\widehat \Psi_{q_1} \ast \widehat \Psi_{q_2} \ast \widehat R_{q_3}\right\|_{X^s_q} 
&\leq C \|\widehat \Psi_{q_1}\|_{Y^s_{q_1}} \|\widehat \Psi_{q_2}\|_{Y^s_{q_2}} \|\widehat R_{q_3} \|_{X^s_{q_3}},
\\
\|\widehat \Psi_{q_1} \ast \widehat R_{q_2} \ast \widehat R_{q_3}\|_{X^s_q} 
&\leq C \|\widehat \Psi_{q_1}\|_{Y^s_{q_1}} \|\widehat R_{q_2} \|_{X^s_{q_2}}\|\widehat R_{q_3} \|_{X^s_{q_3}}, 
\\
\|\widehat R_{q_1} \ast \hat R_{q_2} \ast \hat R_{q_3}\|_{X^s_q} 
&\leq C \|\widehat R_{q_1} \|_{X_{q_1}^s} \|\widehat R_{q_2} \|_{X_{q_2}^s}\|\widehat R_{q_3} \|_{X_{q_3}^s},
\end{align*}
where for $ q = v $ we have $ q_1 = q_2 = q_3 = v $ and where for $ q = w $ at least one of the 
indices $ q_j $ equals $ w $. 
Finally for estimating $ \|\widehat \Psi_q \|_{Y_q^s} $ 
we use 
\begin{align*}
& 2\Big\|\frac{1}{\varepsilon} \widehat A_q \Big( \frac{\cdot -k_0}{\varepsilon}\Big) +
\mathrm{h.o.t.}\Big\|_{Y^s_q} 
\leq C \Big\|\frac{1}{\varepsilon} \widehat A_q \Big( \frac{\cdot}{\varepsilon}\Big)
\Big\|_{Y_q^s}+ \mathrm{h.o.t.} \\
&\leq C \| \widehat A_q\|_{Y^s_q} + \mathrm{h.o.t.} \leq C \| \widehat A_q\|_{X^{s+1}_q} + \mathrm{h.o.t.}.
\end{align*}
Note that we estimated $ \widehat \Psi_q $ in $ Y^s_q $ and not in $ X_q^s $ since 
$\|\widehat \Psi_q\|_{X_q^s} = \mathcal{O} (\varepsilon^{-1/2})$ for $q = w$ which is too large 
to derive estimates on the natural time scale $\mathcal{O}(1/\varepsilon^2)$ with respect to $t$.
\end{proof}

Using the previous lemmas shows that
$$ 
Z(t) = \|\widehat{\mathcal{R}}_v(t)\|_{\ell^2_{s}} + \|\widehat{\mathcal{R}}_w(t)\|_{L^2_{s}}
$$
satisfies
\begin{align*}
Z(t)
&\leq C\varepsilon^2 \int^t_0 \left( Z(\tau)+
  \varepsilon^{\beta- 1} Z(\tau)^2 + \varepsilon^{2\beta-2}
 Z(\tau)^3 + 1\right) d \tau\\
&\leq C\varepsilon^2 \int^t_0 \left( Z(\tau) +
  2 \right) d \tau
\leq 2 C T_0 + C \varepsilon^2 \int^t_0 Z(\tau) d
\tau  
\end{align*}
which holds as long as  
\begin{equation} \label{witt}
\varepsilon^{\beta -1} Z(\tau)^2+
\varepsilon^{2\beta- 2}Z(\tau)^3 \leq 1.
\end{equation}
Applying Gronwall's inequality yields
\[
Z(t) =  \|\widehat{\mathcal{R}}_v(t)\|_{\ell^2_{s}} + \|\widehat{\mathcal{R}}_w(t)\|_{L^2_{s}}\leq 2CT_0 e^{C\varepsilon^2t} \leq 2 CT_0
 e^{CT_0} = M
\]
for all $t\in [0,T_{0}/\varepsilon^2]$. Choosing $\varepsilon_0 > 0$ such that
$\varepsilon_0^{\beta - 1} M^2 + \varepsilon_0^{2\beta - 2} M^3 \leq 1$ 
ensures that condition \eqref{witt} is satisfied. This completes the proof of our approximation result.

\begin{remark}{\rm 
Local existence and uniqueness of solutions to the nonlinear wave equation \eqref{eq:3a1}, as well as to the error equations \eqref{huerrv}-\eqref{huerrw}, hold in the function spaces in which the error estimates were derived.
}\end{remark}

\section{Discussion}

\label{secdisc}

We strongly expect that the existing theory about the validity of NLS approximations, 
i.e., the handling of quadratic nonlinearities by normal form transformations, cf. \cite{Ka87}, the handling of 
resonances, cf. \cite{Schn05}, or the handling of quasilinear systems, cf. \cite{Du17}, can be transferred in the same way.
Therefore, we strongly expect that an approximation result, similar to Theorem \ref{thmain},
holds for the water wave problem as well. This will be the subject of future research.
The present paper goes beyond some formal arguments and 
is a strong indication for the occurrence of 
freak or rogue wave behavior in almost 
all dispersive wave system where the defocusing NLS equation occurs as 
an amplitude equation in the above sense.

The Peregrine soliton is not the only interesting solution where our theory applies, too. 
For illustration we recall some of these solutions from the existing literature.
The Peregrine soliton can be obtained from the family of solutions
$$ 
A(X,T) = e^{iT}\left(1 + \frac{2(1-2a)\cosh(RT) + i R \sinh(RT)}{\sqrt{2a} \cos(\Omega X)- \cosh(RT)} \right)
$$
in the limit $  a \to 1/2 $, cf. \cite{AAT09}, where
$$ 
R = \sqrt{8a(1-2a)}, \qquad \Omega = 2 \sqrt{1-2a}.
$$
For $ 0 < a < 1/2 $ the members of this family  are called Akhmediev Breathers.
They are spatially periodic and localized in time 
above a time-periodic spatially constant background  state.
See  Figure \ref{sec6fig1_a}.
 For $ a > 1/2 $ 
the members of this family are called Kuznetsov-Ma solitons.
They are 
time-periodic and spatially localized  
above a time-periodic  spatially constant background  state.
See  Figure \ref{sec6fig1_b}.
Our approximation result, Theorem \ref{thmain}, easily applies for all members of the Kuznetsov-Ma family, too.

\begin{figure}
  \centering
  \begin{minipage}[b]{0.45\textwidth}
    \includegraphics[width=\textwidth]{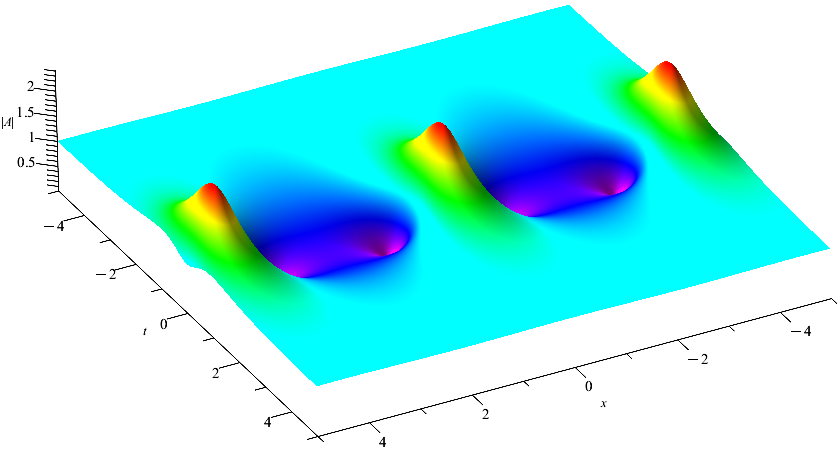}
    \caption{Absolute value of an Akhmediev breather (periodic in space) for $a=\frac{1}{4}$.\label{sec6fig1_a}}
  \end{minipage}
  \hfill
  \begin{minipage}[b]{0.45\textwidth}
    \includegraphics[width=\textwidth]{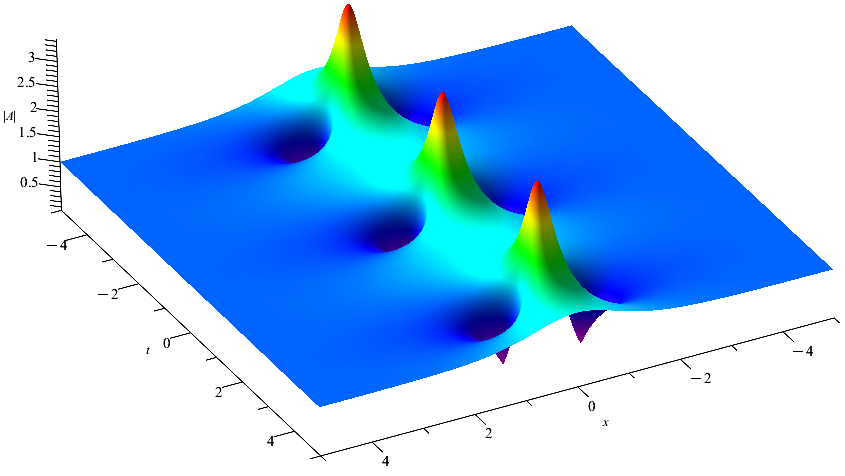}
    \caption{Absolute value of a Kuznetsov-Ma soliton (periodic in time) for $a=\frac{3}{4}$.\label{sec6fig1_b}}
  \end{minipage}
\end{figure}

Finally, the Peregrine soliton can also be seen as 
a member of   a family  of solutions of the form
$$ 
A_j(X,T) = e^{iT} \left((-1)^j + \frac{G_j + i H_j}{D_j} 
\right) 
$$ 
where the $ G_j $ and  $ H_j $ are suitable chosen complex polynomials and the $ D_j $ are 
suitable chosen positive real polynomials, cf. \cite{AAS09}. For the Peregrine soliton we have 
$ j = 1 $, $ G_1 = 4 $, $H_1 = 8X $, and $ D_1 = 1 + 4 X^2  + 4 T^2 $.  
Other members of this family are plotted in  Figure \ref{sec6fig2}.
Our approximation result, Theorem \ref{thmain}, applies to all members of this  family.

\begin{figure}
  \centering
    \includegraphics[width=2.5in]{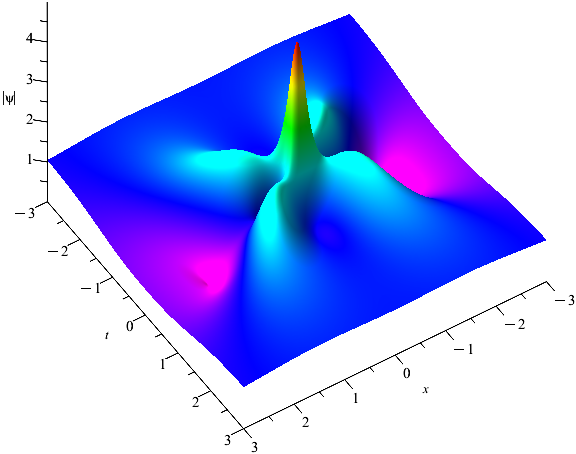}
 \qquad
    \includegraphics[width=2.5in]{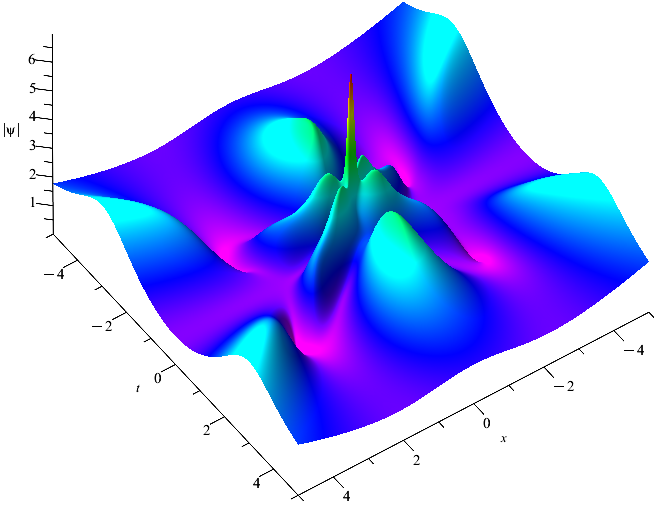}
 \caption{Absolute value of the higher order Peregrine soliton for $j=2,3$.\label{sec6fig2}}
\end{figure}

\bibliographystyle{alpha}
\bibliography{perebib}

\newcommand{\etalchar}[1]{$^{#1}$}
\begin{thebibliography}{GTY{\etalchar{+}}17}

\bibitem[AASC09]{AAS09}
Nail Akhmediev, Adrian Ankiewicz, and J.~M. Soto-Crespo.
\newblock Rogue waves and rational solutions of the nonlinear schr\"odinger
  equation.
\newblock {\em Phys. Rev. E}, 80:026601, Aug 2009.

\bibitem[AAT09]{AAT09}
N.~Akhmediev, A.~Ankiewicz, and M.~Taki.
\newblock Waves that appear from nowhere and disappear without a trace.
\newblock {\em Physics Letters A}, 373(6):675--678, 2009.

\bibitem[DH18]{DH18}
Wolf-Patrick D{\"u}ll and Max He{\ss}.
\newblock Existence of long time solutions and validity of the nonlinear
  {Schr{\"o}dinger} approximation for a quasilinear dispersive equation.
\newblock {\em J. Differ. Equations}, 264(4):2598--2632, 2018.

\bibitem[D{\"u}l17]{Du17}
Wolf-Patrick D{\"u}ll.
\newblock Justification of the nonlinear {Schr{\"o}dinger} approximation for a
  quasilinear {Klein}-{Gordon} equation.
\newblock {\em Commun. Math. Phys.}, 355(3):1189--1207, 2017.

\bibitem[D{\"u}l21]{Du21}
Wolf-Patrick D{\"u}ll.
\newblock Validity of the nonlinear {Schr{\"o}dinger} approximation for the
  two-dimensional water wave problem with and without surface tension in the
  arc length formulation.
\newblock {\em Arch. Ration. Mech. Anal.}, 239(2):831--914, 2021.

\bibitem[GTY{\etalchar{+}}17]{roguebook}
Boling Guo, Lixin Tian, Zhenya Yan, Liming Ling, and Yu-Feng Wang.
\newblock {\em Rogue waves. {Mathematical} theory and applications in physics}.
\newblock Berlin: De Gruyter, 2017.

\bibitem[Kal88]{Ka87}
L.~A. Kalyakin.
\newblock Asymptotic decay of a one-dimensional wave-packet in a nonlinear
  dispersive medium.
\newblock {\em Math. USSR, Sb.}, 60(2):457--483, 1988.

\bibitem[KSM92]{KSM92}
Pius Kirrmann, Guido Schneider, and Alexander Mielke.
\newblock The validity of modulation equations for extended systems with cubic
  nonlinearities.
\newblock {\em Proc. R. Soc. Edinb., Sect. A, Math.}, 122(1-2):85--91, 1992.

\bibitem[Per83]{Pe83}
D.~H. Peregrine.
\newblock Water waves, nonlinear schrödinger equations and their solutions.
\newblock {\em The Journal of the Australian Mathematical Society. Series B.
  Applied Mathematics}, 25(1):16–43, 1983.

\bibitem[Sch05]{Schn05}
Guido Schneider.
\newblock Justification and failure of the nonlinear {Schr{\"o}dinger} equation
  in case of non-trivial quadratic resonances.
\newblock {\em J. Differ. Equations}, 216(2):354--386, 2005.

\bibitem[SSZ15]{SSZ15}
Guido Schneider, Danish~Ali Sunny, and Dominik Zimmermann.
\newblock The {NLS} approximation makes wrong predictions for the water wave
  problem in case of small surface tension and spatially periodic boundary
  conditions.
\newblock {\em J. Dyn. Differ. Equations}, 27(3-4):1077--1099, 2015.

\bibitem[SU17]{SU17book}
Guido Schneider and Hannes Uecker.
\newblock {\em Nonlinear {PDEs}. {A} dynamical systems approach}, volume 182 of
  {\em Grad. Stud. Math.}
\newblock Providence, RI: American Mathematical Society (AMS), 2017.

\bibitem[WC17]{WC17}
C.~Eugene Wayne and Patrick Cummings.
\newblock Modified energy functionals and the {NLS} approximation.
\newblock {\em Discrete Contin. Dyn. Syst.}, 37(3):1295--1321, 2017.

\bibitem[Zak68]{Za68}
V.E. Zakharov.
\newblock Stability of periodic waves of finite amplitude on the surface of a
  deep fluid.
\newblock {\em Sov. Phys. J. Appl. Mech. Tech. Phys}, 4:190--194, 1968.

\end{thebibliography}

\end{document}